\documentclass[12pt, oneside]{amsart}
\usepackage{amssymb,amsfonts, verbatim}
\usepackage[all,arc]{xy}
\usepackage{enumerate}
\usepackage{mathrsfs}
\usepackage{hyperref}
\usepackage[all]{xy}
\usepackage{amsthm}
\usepackage[T1]{fontenc}
\usepackage{comment} 
\usepackage{amssymb, amsmath}
\usepackage{tikz-cd}
\usepackage{MnSymbol}
\usepackage{geometry}
\parskip \baselineskip
\usepackage{graphicx,color}
\hypersetup{
pdftitle={Abelianity of the category of near-vector spaces},
pdfauthor={Zurab Janelidze, Sophie Marques  and Daniella Moore},
}
\usepackage{bm}
\usepackage{cooltooltips}

\makeatletter
\newcommand\ackname{Acknowledgements}
\if@titlepage
  \newenvironment{acknowledgements}{%
      \titlepage
      \null\vfil
      \@beginparpenalty\@lowpenalty
      \begin{center}%
        \bfseries \ackname
        \@endparpenalty\@M
      \end{center}}%
     {\par\vfil\null\endtitlepage}
\else
  \newenvironment{acknowledgements}{%
      \if@twocolumn
        \section*{\abstractname}%
      \else
        \small
        \begin{center}%
          {\bfseries \ackname\vspace{-.5em}\vspace{\z@}}%
        \end{center}%
        \quotation
      \fi}
      {\if@twocolumn\else\endquotation\fi}
\fi
\makeatother

\newtheorem{thm}{Theorem}[section]
\newtheorem{cor}[thm]{Corollary}

\newtheorem{lemma}[thm]{Lemma}

\theoremstyle{definition}
\newtheorem{definition}[thm]{Definition}

\newtheorem{exmps}[thm]{Examples}

\theoremstyle{remark}
\newtheorem{remark}[thm]{Remark}

\newtheorem{corollary}[thm]{Corollary}

\newcommand{\finsub}[0]{{\subseteq}_{\text{\sf{fin}}}}

\makeatletter
\let\c@equation\c@thm
\makeatother
\numberwithin{equation}{section}
\title{Near-linear algebra}

\author{Sophie Marques and Daniella Moore}

\begin{document}

\maketitle
\begin{center}
\rm e-mail: smarques@sun.ac.za

\it
Department of Mathematical Sciences, 
University of Stellenbosch, \\
Stellenbosch, 7600, 
South Africa\\
\&
NITheCS (National Institute for Theoretical and Computational Sciences), \\
South Africa \\ \bigskip

\rm e-mail: 21321264@sun.ac.za

\it
Department of Mathematical Sciences, 
University of Stellenbosch, \\
Stellenbosch, 7600, 
South Africa\\
\end{center} 

\begin{abstract}
In this paper, we demonstrate that the realm of near-vector spaces enables us to address non-linear problems while also providing access to most of the tools that linear algebra offers. We establish fundamental results for near-vector spaces, which serve to extend classical linear algebra into the realm of near-linear algebra. Within this paper, we finalize the algebraic proof that for a given scalar group $F$, any non-empty $F$-subspace that remains stable under addition and scalar multiplication constitutes an $F$-subspace. We prove that any quotient of a near-vector space by an $F$-subspace is itself a near-vector space, along with presenting the First Isomorphism Theorem for near-vector spaces. In doing so, we obtain comprehensive descriptions of the span. By defining linear independence outside the quasi-kernel, we introduce a new concept of basis. We also establish that near-vector spaces are characterized based on the presence of a scalar basis.  \\ \bigskip


\end{abstract}
\newcommand{\sophie}[1]{{\color{red}\sf $\clubsuit\clubsuit\clubsuit$ Sophie: [#1]}}
\newcommand{\dani}[1]{{\color{yellow}\sf $ \spadesuit \spadesuit \spadesuit$ Dani: [#1]}}

\begin{acknowledgements}

We extend our gratitude to Prof.~Janelidze for his invaluable insights and discussions throughout the writing process. Prof.~Janelidze also serves as the Ph.D. supervisor for the second author, enriching the collaborative experience. Our appreciation also goes to Dr.~Boonzaaier for his meticulous review of our paper and valuable feedback. We acknowledge the contributions of the anonymous referee whose constructive comments have significantly improved our work. Additionally, we are indebted to Prof.~Wadsworth for his insightful comments, which have greatly contributed to the quality of this paper.

\end{acknowledgements}
\tableofcontents
\section*{Introduction}\label{d1}
Dr.~André introduced the notion of near-vector spaces (see \cite{Andre}). Prof.~Howell and her students significantly contributed to the study of near-vector spaces (see, for instance, \cite{Howell} and \cite{Howell2}).

The world around us is clearly nonlinear. However, when we step outside the realm of linear algebra, things can become very complex. As we will explore throughout this paper, the world of near-vector spaces allows us to work with nonlinear problems and yet grants access to most of the tools linear algebra offers.

The goal of this paper is to develop a basic theory of near-linear algebra inspired by classical linear algebra. Here, we establish some fundamental results. Given a scalar group $F$ (see Definition \ref{scalar group}), we define $F$-subspaces of $F$-spaces simply by requiring closure under scalar multiplication and addition (see Definition \ref{$F$-subspace}). We then prove that an $F$-subspace of a near-vector space is itself a near-vector space (the first author had already proven this result over division rings, and the proof here directly generalizes this result; meanwhile, it was generalized by different means in \cite{HR22}). Another significant result is that any quotient of a near-vector space by an $F$-subspace is also a near-vector space. This result enables us to establish the First Isomorphism Theorem for near-vector spaces. We explore the notion of linear independence outside the quasi-kernel; this notion allows us to define near-vector spaces in terms of the existence of a scalar basis, resulting in an intriguing new notion of basis.

In the first section of the paper, we introduce the preliminary material. We propose a new way to define near-vector spaces by introducing the notion of a scalar group and a scalar group action. This alternative perspective on near-vector spaces will make it easier for us to manipulate them. We also review a few essential elementary properties that will be of significance throughout the paper.

Section 2 focuses on the span, basis, and linear independence of near-vector spaces. We develop fundamental tools to study near-vector spaces, identifying differences between classical linear algebra and near-linear algebra by exploring the properties of span and linear independence. We regard elements in the quasi-kernel as scalar elements, shedding more light on the importance of the definition of the quasi-kernel in near-linear algebra theory (see Definition \ref{quasikernel}). We provide different characterizations of linear independence outside the quasi-kernel (see Lemma \ref{QL}). We introduce two notions of a basis: an $F$-basis and a scalar $F$-basis (see Definition \ref{basis}). These bases encapsulate some of the fundamental distinctions between classical linear algebra and near-linear algebra.

In section 3, we present the key technical results of the paper (see Theorem \ref{minimal} and Corollary \ref{spanlemma}). The first result explicitly describes the span of any vector as a direct sum of elements from the quasi-kernel. The formula involves the notion of the dimension of an element. The proof provides a constructive way to obtain such an element as a direct sum of quasi-kernel elements. An application of this result leads to proving that we can construct a scalar $F$-basis (see Lemma \ref{generatingset}). An important observation is that, unlike classical linear algebra, not all $F$-bases have the same cardinality. Particularly, we prove the Replacement Theorem for scalar $F$-bases (see Lemma \ref{replacement}). Another fundamental result of this section asserts that a near-vector space over a scalar group $F$ is simply a pair $(V, \mu)$ where $V$ is an abelian group and $\mu$ is a scalar group action admitting a scalar $F$-basis (see Theorem \ref{scalarbasis}). The second key technical result is trivial in classical linear algebra but nontrivial in near-linear algebra; it compares the spans of non-zero elements $v$ and $w$, where $w \in \operatorname{Span}(v)$.

In section 4, the work done in the previous section leads to four results that hold true in classical linear algebra and remain valid for near-vector spaces. The first result is that any $F$-subspace is a near-vector space (see Theorem \ref{nvs}). The second result is that any quotient by an $F$-subspace is also a near-vector space (see Theorem \ref{quotientnvs}). From this, we deduce the one-to-one correspondence between the $F$-subspace and the kernel of a linear map of near-vector spaces (see Corollary \ref{$F$-subspacekernel}). Finally, we establish the First Isomorphism Theorem for near-vector spaces (see Theorem \ref{fit}).

Section 5 discusses how any near-vector space can be viewed as a module over a scalar group algebra. The notion of span for near-vector spaces coincides with the span concept with respect to this module. This description offers more insight into near-linear algebra. We derive a geometric interpretation of the quasi-kernel at the section's conclusion.

\section{Preliminary material}
Prof.~Janelidze conceptualized the initial version of the following definition. This definition will aid in the manipulation of near-vector spaces more effectively. 
\begin{definition}
\label{scalar group}
A {\bf scalar group} $F$ is a tuple $F=(F,\cdot,1,0,-1)$ where $(F,\cdot,1)$ is a monoid, $0, -1\in F$, 
$0\cdot \alpha=0=\alpha\cdot 0$ for all $\alpha\in F$, $\{ \pm 1\}$ is the solution set of the equation $x^2=1$, and $(F \backslash \{ 0\},\cdot,1)$ is a group. For all $\alpha \in F$, we denote $-\alpha$ as the element $(-1)\cdot \alpha$.
\end{definition}
\begin{remark}
A monoid formed by adding a zero element to a group is often referred to as a "group with zero". It's worth noting that Definition \ref{scalar group} essentially defines a group with zero in which the equation $x^2=1$ possesses precisely two distinct solutions.
\end{remark}
We can define the concept of the action of a scalar group as follows.

\begin{definition}  
\label{actiondef}
Let $V=(V,+)$ be an abelian group, $(F, \cdot , 1, 0, -1)$ be a scalar group, and $\mu : F \times V \rightarrow V $ be a map that assigns $(\alpha , v)$ to $\alpha \cdot v$.

\begin{enumerate}
\item We denote the map $\mu$ as an {\bf action of $F$ on $V$}. In this paper, when we refer to an action as $\mu$, we will use $\alpha \cdot v$ to denote the image of an element $(\alpha , v)$ in $F \times V$ through $\mu$.
\item We define an action $\mu$ as a {\bf left semigroup action} if, for all $\alpha , \beta \in F$, and $v\in V$, it holds that $\alpha \cdot (\beta \cdot v) = (\alpha \cdot \beta) \cdot v$ and $1 \cdot v = v$.
\item We define a semigroup action $\mu$ as {\bf free} if the action is free, meaning that for any $\alpha, \beta \in F$, and $v \in V$, if $\alpha \cdot v = \beta \cdot v$, then it follows that $v=0 \ \textrm{or} \ \alpha = \beta $.
 \item We define an action $\mu$ as {\bf compatible with the $\mathbb{Z}$-structure of $V$} if the following conditions are met:
 \begin{itemize} 
 \item $\mu$ {\bf acts by endomorphisms}, meaning that for all $\alpha \in F$, $v, w \in V$, the equation $\alpha \cdot (v + w) = \alpha \cdot v + \alpha \cdot w$ holds, 
\item {\bf $-1$ acts as $-id$}, which implies that for all $v \in V$, $-1 \cdot v = -v$, 
\item {\bf $0$ acts trivially}, meaning that for all $v \in V$, $0 \cdot v = 0$.
\end{itemize}
\item We define an action $\mu$ as a {\bf left scalar group action} if it is a left semigroup action compatible with the $\mathbb{Z}$-structure of $V$.
\end{enumerate} 
\end{definition}

\begin{definition}
\label{$F$-subspace}
Let $F$ be a scalar group. An {\bf $F$-space} is a pair $(V, \mu)$ where $V$ is an abelian additive group and $\mu: F \times V \rightarrow V$ is a left scalar group action.
When there is no confusion, we will simply denote $(V, \mu)$ as $V$ and $\mu(\alpha, v)$ as $\alpha \cdot v$, for all $\alpha \in F$ and $v \in V$. We say that $W$ is an {\bf $F$-subspace} of $V$ if $W$ is a nonempty subset of $V$ that is closed under addition and scalar multiplication.
\end{definition}
Next, we define the notion of the quasi-kernel of an $F$-space.
\begin{definition}
\label{quasikernel}
Let $V$ be an $F$-space. We define the {\bf quasi-kernel} of $V$ to be
$$Q(V) = \{v \in V \mid \forall_{\alpha, \beta \in F} \exists_{\gamma \in F} [\alpha \cdot v + \beta \cdot v = \gamma \cdot v]\}.$$
\end{definition}

We can also define a semigroup action on the right. In this paper, we will restrict ourselves to left near-vector spaces. All of our results will also apply to right near-vector spaces. The definition below is equivalent to \cite[Definition 4.1]{Andre} of a near-vector space using the terminologies introduced in Definition \ref{actiondef}.
\begin{definition}
\label{NVS2}
A {\bf left near-vector space} over a scalar group $F$ is an $F$-space $(V, \mu)$ such that the left scalar group action $\mu$ is free and $Q(V)$ generates $V$ seen as an additive group.
Any trivial abelian group has a near-vector space structure through the trivial action. We refer to such a space as a {\bf trivial near-vector space over $F$}. We denote a trivial near-vector space as $\{0\}$.
\end{definition}

In the following, a left near-vector space is simply referred to as a near-vector space and $F$ denotes a scalar group and $V$ denotes a near-vector space over the scalar group $F$ unless stated differently.
\begin{remark}
\label{QK}
\begin{enumerate}
\item The freeness of the action $\mu$ is usually referred to as the fixed point-free property in near-vector space theory literature.
\item The fixed point-free property can be translated in terms of the stabilizer as follows. For all $v \in V$, we have that the set of elements of $F$ that stabilize $v$ is described as follows:
\begin{equation*}
    \operatorname{Stab}_F(v) = \begin{cases} F & \text{if } v = 0; \\
    \{1\} & \text{if } v \neq 0.
    \end{cases}
\end{equation*}
\item $Q(V)$ is stable by scalar multiplication (see \cite[Lemma 2.2]{Andre}).
\end{enumerate}
\end{remark}

We provide illustrative examples of near-vector spaces that deviate from traditional vector spaces, highlighting their unique characteristics.

\begin{exmps}
\label{nearfield}
We provide three examples of near-vector spaces. The first example involves a near-vector space over a scalar group that constitutes the underlying multiplicative group of a field. The second example pertains to a near-vector space over a scalar group that is the underlying multiplicative group of a near-field. We denote $[a]_n$ the equivalence class of the integer $a$ modulo $n$ or simply $[a]$ when $n$ is clear from context.

\begin{enumerate}
\item Consider the scalar group $(\mathbb{R}, \cdot)$, where $1$ serves as the identity element, $-1$ serves as the additive inverse, $0$ serves as the zero element, and multiplication is the usual multiplication in $\mathbb{R}$. This scalar group is the underlying multiplicative group of the field $(\mathbb{R}, +, \cdot)$. We can define an action by endomorphism of $\mathbb{R}$ on $\mathbb{R}^3$  as follows:
 $$\alpha \smallstar (x,y,z) = (\alpha x, \alpha y, \alpha^{3} z)$$ for $\alpha \in \mathbb{R}$ and $(x,y,z) \in \mathbb{R}^{3}$. $(\mathbb{R}^3,\smallstar)$ becomes a near-vector space with these operations. Notably,
 \begin{equation*}
     Q(\mathbb{R}^3,\smallstar) = \left\{\alpha(1,1,0)\mid \alpha \in \mathbb{R} \right\} \cup \left\{\alpha(0,0,1)\mid \alpha \in \mathbb{R} \right\},
\end{equation*}
thus making the quasi-kernel of $(\mathbb{R}^3,\smallstar)$ the union of two $\mathbb{R}$-subspaces: $\{\alpha (1,1,0) \mid \alpha \in \mathbb{R}\}$ and $\{\alpha (0,1,1) \mid \alpha \in \mathbb{R}\}$.
\item Consider the scalar group $ (\mathbb{Z}_{5}, \cdot)$, where $[1]$ serves as the identity element, $[4]$ serves as the additive inverse, $[0]$ serves as the zero element, and multiplication is the usual multiplication in $\mathbb{Z}_{5}$. This scalar group is the underlying multiplicative group of the field $(\mathbb{Z}_{5}, +, \cdot)$. 
We define an action by endomorphism of $\mathbb{Z}_{5}$ on $\mathbb{Z}_{5}^{2}$ as follows:
\begin{equation*}
    \alpha \star(x,y) = (\alpha x, -\alpha y)
\end{equation*}
for $\alpha \in \mathbb{Z}_{5}$ and $(x,y) \in \mathbb{Z}_{5}^{2}$.
$(\mathbb{Z}_{5}^{2},\star)$ becomes a near-vector space with these operations. Notably,
\begin{equation*}
    Q(\mathbb{Z}_{5}^2, \star) = \{\alpha ([1],[0]) \mid \alpha \in \mathbb{Z}_5\}\cup \{\alpha ([0],[1])\mid \alpha \in \mathbb{Z}_5\},
\end{equation*}
thus making the quasi-kernel of $(\mathbb{Z}_{5}^2, \star)$ the union of two $\mathbb{Z}_{5}$-subspaces: $\{\alpha ([1],[0]) \mid \alpha \in \mathbb{Z}_5\}$ and $\{\alpha ([0],[1])\mid \alpha \in \mathbb{Z}_5\}$. 
However, $(\mathbb{Z}_{5}^2, \star)$ is not a vector space, as evidenced by the fact that $([1],[2])([1],[1]) = ([3],[2]) \neq ([3],3) = [1]([1],[1])+[2]
([1],[1])$.

\item Consider the finite field $(GF(3^{2}), +, \cdot)$ (see also \cite{pilz}). We have
\begin{equation*}
    GF(3^{2}) := \{[0],[1],[2], \gamma, [1]+ \gamma, [2] + \gamma, [2]\gamma, [1]+[2]\gamma, [2] + [2]\gamma \},
\end{equation*}
where $\gamma$ is a root of $x^{2}+[1] \in \mathbb{Z}_{3}[x]$. The operations on $GF(3^{2})$ are defined as follows:
\begin{equation*}
\label{addition}
   ([a]+[b]\gamma)+([c]+[d]\gamma) = ([a]+[c])+ ([b]+[d])\gamma,
\end{equation*}
for all $[a],[b],[c],[d] \in \mathbb{Z}_{3}$, and the usual multiplication is depicted in Table~\ref{TableA}.

{\tiny
\begin{table}[h]
\centering 
\begin{tabular}{l|cccccccccccccccccccccccccc} 
 $\cdot$ &$[0]$ & $[1]$ &$[2]$& $\gamma$ &$[1]+\gamma$ &$[2]+\gamma$ &$[2]\gamma$ & $[1]+[2]\gamma$ &$[2]+[2]\gamma$ &
\\ \\
\hline \\
 $[0 ]$ & $[0 ]$&$[0]$ & $[0 ]$ &$[ 0]$ & $[0]$ & $[0]$ & $[0]$& $[0]$& $[0]$& \\ \\
 $[1 ]$ &  $[0 ]$ & $[1 ]$&  $[2]$ & $\gamma$ & $[1 ]+ \gamma$ & $[2]+\gamma$ & $[2]\gamma$ & $[1 ]+[2]\gamma$ & $[2]+[2]\gamma$ \\ \\

$[2]$ & $[0 ]$ & $[2]$ & $[1 ]$ & $\gamma$ & $[2]+[2]\gamma$ & $1+[2]\gamma$ & $\gamma$ & $[2]+\gamma$ & $[1 ]+\gamma$ \\ \\

$\gamma$ &$[0 ]$& $\gamma$ & $[2]\gamma$ & $[2]$ & $[2]+\gamma$ & $[2]+[2]\gamma$ & $[1 ]$& $[1 ]+\gamma$& $[1 ]+[2]\gamma$ \\ \\

$[1 ]+\gamma$ &$[0 ]$ & $[1 ]+\gamma$ & $[2]+[2]\gamma$ & $[2]+\gamma$ & $[2]\gamma$ & $[1 ]_3$ & $[1 ]_3+[2]_3\gamma$ & $[2]_3$ & $\gamma$ \\ \\

$[2]+\gamma$ &$[0 ]$ & $[2]+\gamma$ & $[1 ]+[2]\gamma$ & $[2]+[2]\gamma$ & $[1 ]$ & $\gamma$ & $[2]+\gamma$ & $[2]\gamma$ & $[2]$ \\ 
\\

$[2]\gamma$ &$[0 ]$ & $[2]\gamma$ & $\gamma$ & $[1 ]$& $[1 ]+[2]\gamma$& $[1 ]+\gamma$& $[2]$ & $[2]+[2]\gamma$& $[2]+\gamma$ \\ \\

$[1 ]+[2]\gamma$ &$[0 ]$ & $[1 ]+[2]\gamma$& $[2]+\gamma$& $[1 ]+\gamma$& $[2]$ & $[2]\gamma$& $[2]+[2]\gamma$& $\gamma$& $[1 ]$ \\ \\
$[2]+[2]\gamma$ &$[0 ]$& $[2]+[2]\gamma$& $[1 ]+\gamma$& $[2]+\gamma$& $[1 ]$&$[2]$ &$[1 ]+[2]\gamma$&$[2]\gamma$&$\gamma$ \\ \\ \\
\end{tabular}
\caption{Multiplication table for the field $GF(3^2)$}\label{TableA}
\end{table}}

We introduce a new operation $\circ$ on $GF(3^{2})^2$ as follows:
\begin{equation*}
    x \circ y := \left\{ \begin{array}{ll}
     x \cdot y & \text{if} \ x \ \text{is a square in } (GF(3^{2}), +, \cdot) \\
     x\cdot y^{3} &  \text{otherwise}. 
\end{array}
\right.
\end{equation*}
This operation is also presented in Table~\ref{TableB}. We have that $(GF(3^{2}), +, \circ)$ is a near-field and  $(GF(3^2),\circ)$ a scalar group.
{\tiny
\begin{table}[h]
\centering 
\begin{tabular}{l|cccccccccccccccccccccccccc} 

 $\circ$ & $[0]$ & [1] &[2]& $\gamma$ &$[1]+\gamma$ &$[2]+\gamma$ &$[2]\gamma$ & $[1]+[2]\gamma$ &$[2]+[2]\gamma$ &
\\  \\
\hline  \\
 $[0]$ &    $[0]$ &  $[0]$&   $[0]$ &  $[0]$&   $[0]$&   $[0]$ &   $[0]$&  $[0]$&   $[0]$& \\ \\
$[1]$ &   $[0]$ & $[1]$ & $[2]$ & $\gamma$ & $ [1]+ \gamma$ & $[2]+\gamma$ & $[2]\gamma$ & $ [1]+[2]\gamma$ & $[2]+[2]\gamma$ \\ \\

$[2]$ &  $[0]$ & $[2]$ & $[1]$  & $[2]\gamma$ & $[2]+[2]\gamma$ & $ [1]+[2]\gamma$ & $\gamma$ & $[2]+\gamma$ & $ [1]+\gamma$ \\ \\

$\gamma$ &  $[0]$ & $\gamma$ & $[2]\gamma$ & $[2]$ & $[2]+\gamma$ & $[2]+[2]\gamma$ &  $[1]$ & $ [1]+\gamma$& $ [1]+[2]\gamma$ \\ \\

$[1]+\gamma$ & $[0]$ & $ [1]+\gamma$ & $[2]+[2]\gamma$ & $ [1]+[2]\gamma$ & $[2]$ & $\gamma$ & $[2]+\gamma$ & $[2]\gamma$ &  [1] \\ \\

$[2]+\gamma$ & $[0]$ & $[2]+\gamma$ & $ [1]+[2]\gamma$ & $ [1]+\gamma$ & $[2]\gamma$ & $[2]$ & $[2]+[2]\gamma$&  $[1]$  & $\gamma$ \\ \\

$[2]\gamma$ & $[0]$ & $[2]\gamma$ & $\gamma$ &  $[1]$ & $ [1]+[2]\gamma$& $ [1]+\gamma$& $[2]$& $[2]+[2]\gamma$& $[2]+\gamma$ \\ \\

$ [1]+[2]\gamma$ & $[0]$& $ [1]+[2]\gamma$& $[2]+\gamma$& $[2]+[2]\gamma$& $\gamma$&  $[1]$ & $ [1]+\gamma$& $[2]$& $[2]\gamma$ \\ \\
$[2]+[2]\gamma$& $[0]$& $[2]+[2]\gamma$& $ [1]+\gamma$& $[2]+\gamma$&  $[1]$ &$[2]\gamma$ &$ [1]+[2]\gamma$& $\gamma$& $[2]$ \\ \\
\end{tabular}
\caption{Multiplication table for the near-field $GF(3^2)$}\label{TableB}
\end{table}}

By defining an action by endomorphism $\diamond$ of $GF(3^2)$ on $(GF(3^2))^2$ as 
$$\alpha \diamond (x_{1},x_{2}) = (\alpha \circ x_{1}, \alpha \circ x_{2}),$$ we create the near-vector space $((GF(3^2))^2,\diamond)$. Additionally,
$$Q((GF(3^2))^2,\diamond) = \{ \lambda (d_{1},d_{2}) \mid \lambda \in GF(3^2) \text{ and } d_{1},d_{2} \in GF(3)\}.$$
For more detailed proof, refer to \cite[Section 2.3]{Moore}. However, this is not a vector space, simply noting that $([1]+\gamma)([1],[2]+\gamma) = ([1]+\gamma, \gamma) \neq ([1]+\gamma,[1]) = [1]([1],[2]+\gamma)+\gamma([1],[2]+\gamma)$.
\end{enumerate}
We will utilize Example \ref{nearfield} (1) to elucidate the properties of near-vector spaces throughout this paper. Analogous observations can be drawn from the remaining examples. The analysis of the other two instances provided in Example \ref{nearfield} is left to the readers.

\end{exmps}

The concept of the quasi-kernel enables the introduction of an abelian group operation on \(F\). This operation, outlined in the ensuing definition, transforms \(F\) into a near-field (see \cite[Lemma 2.4]{Andre}).

\begin{definition}\cite[Section 2]{Andre} \label{qvs} For any \(v \in Q(V)\backslash \{ 0\}\) and \(\alpha , \beta \in F\), we denote \(\alpha +_v \beta\) as the unique \(\gamma \in F\) such that \(\alpha \cdot v + \beta \cdot v = \gamma \cdot v\). Given a family \((\alpha_i)_{i \in \{1, \cdots , n\}}\) of elements in \(F\), we signify by \({}^v \sum_{i=1}^n\alpha_i\) the sum \(\alpha_1 +_v \alpha_2 +_v \cdots +_v \alpha_n\). 
\end{definition}

\begin{remark}
\label{alpha+vbeta=0} 
For \(\alpha, \beta \in F\) and \(v \in Q(V) \backslash \{0\}\), it holds that \(\alpha +_{v} \beta = 0\) if and only if \(\alpha = -\beta\). This equivalence stems from the following reasoning: Let \(\alpha, \beta \in F\) and \(v \in Q(V) \backslash \{0\}\). Assume \(\alpha+_{v} \beta = 0\). Then \((\alpha +_{v} \beta)\cdot v = \alpha \cdot  v + \beta \cdot  v = 0\), implying that \(\alpha \cdot v = -\beta \cdot v\). Due to the fixed point free property of \(V\), it follows that \(\alpha = -\beta\). Conversely, let \(\alpha = - \beta\). This leads to \(\alpha \cdot v + \beta \cdot v = \alpha \cdot v - \alpha \cdot  v = 0\). Since \(v \in Q(V)\), we have \(\alpha \cdot v + \beta \cdot v = (\alpha +_{v} \beta) \cdot v\). Consequently, we have \((\alpha +_{v} \beta) \cdot v = 0 = 0 \cdot v\), and \(\alpha +_{v} \beta = 0\) by the fixed point free property of \(V\).
\end{remark}

In the realm of near-vector space theory, a new concept for the dimension of an element emerges.

\begin{definition} \cite[Definition 3.5]{Howellspanning}
\label{summ}
 For \(v \in V \backslash \{ 0\} \), the \textbf{dimension of an element \(v\)}, denoted by \(\operatorname{dim}(v)\), is the minimum number of distinct elements in \(Q(V)\) such that \(v\) can be expressed as a sum of those elements. For \(v=0\), we set \(\operatorname{dim}(v)=0\).
\end{definition}

Although elementary once stated, the following lemma is pivotal in proving the main results.

\begin{lemma}
\label{pluses1}
If \(A\) is a non-empty finite subset of \(Q(V)\), and \(+_{a} = +_{b}\) for all \(a, b \in A\), then \(\sum_{a \in A} a \in Q(V)\).
\end{lemma}

\begin{proof}
Let \(\alpha, \beta \in F\). Assume \(+_{a} = +_{b}\) for all \(a, b \in A\). Let \(+ := +_a\) for all \(a \in A\). Then, 
\begin{align*}
    \alpha \cdot \sum_{a \in A} a + \beta \cdot \sum_{a \in A}a & = \sum_{a \in A} \alpha \cdot a + \sum_{a \in A}\beta \cdot  a = \sum_{a \in A}(\alpha +_{a} \beta) \cdot a \\
    &= \sum_{a \in A}(\alpha + \beta)\cdot a  = (\alpha + \beta)\cdot \sum_{a \in A} a 
\end{align*}
and thus, \(\sum_{a \in A} a \in Q(V)\).
\end{proof}

\begin{definition}
We write \(A \finsub B\) to indicate that \(A\) is a finite subset of \(B\). 
\end{definition}
The subsequent lemma, while relatively straightforward to prove once introduced, plays a fundamental role in the underlying concepts of the main results.

\begin{lemma}
\label{pluses}
Let $v \in V$. There is $\Theta \finsub Q(V)$ with $|\Theta|=dim(v)$ such that $v = \sum_{q \in \Theta } q$ (see Definition \ref{summ}). Then, for all $q, q' \in \Theta$ such that $q \neq q'$, we have $+_{q} \neq +_{q'}$. 
\end{lemma}

\begin{proof}
Let \(v \in V\). Let \(\Theta \finsub Q(V)\) with \(|\Theta|=\operatorname{dim} (v)\) such that \(v = \sum_{q \in \Theta } q\). We proceed by contradiction. Assume there exist \(q, q' \in \Theta\) with \(q \neq q'\) and \(+_{q} = +_{q'}\). By Lemma \ref{pluses1}, \(q+q' \in Q(V)\), contradicting the minimality of \(\operatorname{dim}(v)\).
\end{proof}

\begin{remark}
\label{Q(W)subsetQ(W')}
\begin{enumerate} 
\item If \(W\) and \(W'\) are \(F\)-subspaces of \(V\) with \(W \subseteq W'\), then \(Q(W) = W\cap Q(W')\). Particularly, \(Q(W) = W\cap Q(V)\) and \(Q(W) \subseteq Q(W')\).
\item When \(W\) is an \(F\)-subspace of \(V\), the action \(\mu\) that imparts \(V\) with its near-vector space structure over the scalar group \(F\) also induces a left scalar group free action \(\mu\vert_{W}: F \times W \rightarrow W\) by restricting the map \(\mu\) to \(F \times W\). At this stage, it may not be evident that the quasi-kernel of \(W\) generates \(W\) as an abelian group. We will prove this in Theorem \ref{nvs}. 
\end{enumerate} 
\end{remark}

To describe spans of near-vector spaces, the subsequent terminology is required.

\begin{definition}
\label{sum$F$-subspace} 
For an \(F\)-space \(V\), let \(\{ S_i\}_{i\in I}\) be a family of subsets of \(V\). We define the \textbf{sum} of the family \(\{ S_i\}_{i\in I}\) as follows:
\[
\sum_{i \in I} S_i = \left\{ \sum_{j \in J} s_j \mid J \finsub I, \ s_j \in S_j , \ \text{for all } j \in J \right\}.
\]
The empty sum is denoted as \(\sum_{v \in \emptyset}v = 0\). If for every \(J \finsub I\), \(s_j,s_{j}' \in S_j\) for all \(j\in J\) such that \(\sum_{j\in J} s_j= \sum_{i\in J} s_j'\), then \(\sum_{i \in I} S_i\) is termed a \textbf{direct sum}, denoted as \(\oplus_{i \in I} S_i\). The family \(\{ S_i\}_{i\in I}\) is referred to as \textbf{additively independent} when the sum \(\sum_{i \in I} S_i\) is a direct sum.  
\end{definition}

Demonstrating that the sum of subspaces remains a subspace, as stated by the following lemma, is a straightforward task.

\begin{lemma}
\label{sumissubspace}
Consider a family \(\{ W_i\}_{i\in I}\) of \(F\)-subspaces of \(V\). The set \(\sum_{i \in I} W_i\) is likewise an \(F\)-subspace of \(V\).
\end{lemma}

The First Isomorphism Theorem for near-vector spaces will be established. The ensuing definitions introduce the concepts of homomorphism, kernel, and image in this context. It's important to note that these definitions mirror those used in classical linear algebra.

\begin{definition} \label{ker} 
Given near-vector spaces \(V\) and \(V'\) over the same scalar group \(F\), a morphism of near-vector spaces \(f: V \rightarrow V'\) is referred to as a \textbf{linear map}. In other words, for all \(x,y \in V\) and \(\alpha \in F\):
\begin{enumerate}
    \item \(f(x + y) = f(x)+f(y)\);
    \item \(f(\alpha \cdot x) = \alpha \cdot f(x)\).
\end{enumerate}
\end{definition}

Throughout this paper, we assume that \(f: V \rightarrow V'\) is a linear map, with \(V\) and \(V'\) consistently denoting near-vector spaces over the same scalar group \(F\).

\begin{definition} \label{ker} 
Given a linear map \(f: V \rightarrow V'\), we define the \textbf{kernel of \(f\)} as the set 
\[
\operatorname{Ker}(f)= \{ v \in V \mid f(v) =0\},
\]
and the \textbf{image of \(f\)} as the set 
\[
\operatorname{Im}(f) = \{ f(v)\mid v \in V\}.
\]
\end{definition}

Similar to the principles in classical linear algebra, we can establish that the kernel and image of a linear map also conform to the structure of an \(F\)-space. 

\begin{lemma}
For a linear map \(f:V \rightarrow V'\), the set \(\operatorname{Ker}(f)\) is an \(F\)-subspace of \(V\), and \(\operatorname{Im}(f)\) is an \(F\)-subspace of \(V'\).
\end{lemma}

\section{Linear independence, span and basis} 

The concept of the span for \(F\)-spaces aligns closely with the definition of span for vector spaces. Here, we introduce the notion of span for any \(F\)-space. It's important to note that the fixed point free property and the quasi-kernel have no impact on the following definitions.

\begin{definition}\cite[Definition 3.2]{Howellspanning}
Consider an \(F\)-space \(V\). We define the {\bf span of a set \(S\)} as the intersection \(W\) of all non-empty subsets of \(V\) that are closed under addition and scalar multiplication and contain the set \(S\). This span is denoted as \(\operatorname{Span}(S)\).
\end{definition}

\begin{remark}
\label{smallest}
For an \(F\)-space \(V\) and a subset \(S \subseteq V\), the following observations hold:
\begin{enumerate} 
\item If \(A \subseteq V\) and \(A \subseteq \operatorname{Span}(S)\), then \(\operatorname{Span}(A) \subseteq \operatorname{Span}(S)\), given that \(\operatorname{Span}(A)\) is the smallest \(F\)-subspace containing \(A\) and \(\operatorname{Span}(S)\) is a non-empty subset of \(V\) closed under addition and scalar multiplication that contains \(S\).
\item When \(S = \emptyset\), it follows that \(\operatorname{Span}(S) = \{0\}\).
\item For all \(\alpha \in F \backslash \{ 0 \}\) and \(v \in V\), \(\operatorname{Span} (\alpha \cdot v) = \operatorname{Span} (v)\).
\end{enumerate}
\end{remark}

Just as in classical linear algebra, the span of a set in \(F\)-spaces corresponds precisely to the set of linear combinations of its elements. The only distinction lies in the complexity of the linear combinations.

\begin{lemma}
\label{spanlin}
Let \(V\) be an \(F\)-space, and \(S\subseteq V\). Then, the set \(\operatorname{Span} (S)\) can be expressed as follows:
\begin{align*}
     \operatorname{Span} (S)= & \left\{\sum_{a \in A} \left(\sum_{i=1}^{n_a} \alpha_{a,i} \cdot a \right) \middle| A \finsub S, n_a \in \mathbb{N}, \alpha_{a,i} \in F, \forall a\in A, \forall i \in \{ 1, \cdots ,n_a\}  \right\}\\
= & \sum_{s\in S}   \operatorname{Span} (s).
\end{align*}
\end{lemma}
\begin{proof}
Let \(S \subseteq V\) and define
\begin{equation*}
     L (S):= \left\{\sum_{a \in A} \left(\sum_{i=1}^{n_a} \alpha_{a,i} \cdot a\right) \middle| A \finsub S, n_a \in \mathbb{N}, \alpha_{a,i} \in F, \forall a \in A, \ \forall i \in \{ 1, \cdots ,n_a\} \right\}.
\end{equation*}
We aim to show that \(L(S) \subseteq \operatorname{Span}(S)\). By the definition of \(\operatorname{Span}(S)\), it's given that \(S \subseteq \operatorname{Span}(S)\). Furthermore, since \(\operatorname{Span}(S)\) represents the intersection of all non-empty subsets of \(V\) that are closed under addition and scalar multiplication and contain \(S\), it also satisfies these properties, thus making it closed under linear combinations.
Thus, \(L(S) \subseteq \operatorname{Span}(S)\), as desired. For the reverse inclusion, we observe that \(L(S)\) contains \(S\) and is closed under addition and scalar multiplication. This yields \(\operatorname{Span}(S) \subseteq L(S)\), considering that \(\operatorname{Span}(S)\) is the smallest non-empty subset of \(V\) closed under addition and scalar multiplication and containing \(S\). As \(\operatorname{Span}(S)\) is a non-empty subset of \(V\) closed under addition and scalar multiplication and containing \(S\), we deduce that \(\sum_{s \in S}\operatorname{Span}(s) \subseteq \operatorname{Span}(S)\). Since \(S \subseteq \sum_{s \in S}\operatorname{Span}(s)\), we have, by applying Remark \ref{smallest}, \(\operatorname{Span}(S) \subseteq \sum_{s \in S}\operatorname{Span}(s)\). As a result, we obtain \(\operatorname{Span}(S) = \sum_{s \in S}\operatorname{Span}(s)\).
\end{proof} 

\begin{definition}
Consider an \(F\)-space \(V\) and \(S\subseteq V\). An element of the form \(\sum_{a \in A} \left(\sum_{i=1}^{n_a} \alpha_{a,i} \cdot a \right)\) according to the notations in Lemma \ref{spanlin} is referred to as a {\bf linear combination of elements of \(S\)}.
\end{definition}

\begin{remark}
For an \(F\)-space \(V\) and \(s\in V\), the set \(\operatorname{Span} (s)\) can be described as:
\[
\operatorname{Span} (s)= \left\{ \sum_{i=1}^{n} \alpha_{s,i} \cdot s \middle|  n \in \mathbb{N}, \alpha_{s,i} \in F \text{ and } i \in \{ 1, \cdots ,n\} \right\}.
\]
\end{remark}

Next, we delve into the definition of linear independence beyond the context of the quasi-kernel.

\begin{definition}
\label{linindep} 
In an \(F\)-space \(V\), let \(S \subseteq V\). \(S\) is considered {\bf linearly independent} if \(0 \notin S\) and for any non-empty subset \(A \finsub S\), along with \(n_a\in \mathbb{N}\) and \(\alpha_{a,i} \in F\) where \(a \in A\) and \(i \in \{ 1, \cdots , n_a\}\), if \(\sum_{a \in A} \left( \sum_{i=1}^{n_a} \alpha_{a,i} \cdot a \right) = 0\), then \(\sum_{i=1}^{n_a} \alpha_{a,i}\cdot a = 0\) for all \(a \in A\). An element \(v\) is {\bf scalar} if \(\operatorname{Span} (v) = F \cdot v\).  
\end{definition}

\begin{remark}
\label{linindepinQ(V)}
\begin{enumerate}
    \item If \(S \subseteq Q(V)\), then we show in Lemma \ref{L=sum} that the definition of linear independence in \ref{linindep} is equivalent to the standard definition of linear independence in near-vector space theory (see \cite[Theorem 3.1]{Andre}). 
    \item Clearly, \(v\) is scalar if and only if for all \(w \in V\) such that \(F \cdot w \subseteq \operatorname{Span} (v)\), we have \(w \in F \cdot v\).
    \item Note that by definition, \(\emptyset\) is a linearly independent set. 
    \end{enumerate}
\end{remark}

We can now establish that the concept of a scalar element is tantamount to asserting that the element belongs to the quasi-kernel, provided \(V\) is a near-vector space.

\begin{lemma}\label{scalar}
The following assertions are equivalent. 
\begin{enumerate} 
\item \(v\) is scalar;
\item \(F\cdot v = F\cdot w \), for all \(w\in \operatorname{Span} (v)\backslash \{ 0 \} \);
\item \(v \in Q(V)\).
\end{enumerate} 
\end{lemma}
\begin{proof} 
\text{(1) \(\Rightarrow\) (2):} Let \(w\in \operatorname{Span} (v)\backslash \{ 0 \}\). Then, by assumption, \(w \in F\cdot v\), implying the existence of \(\alpha \in F \backslash \{0\}\) such that \(w = \alpha \cdot v\). Consequently, \(F\cdot w = F\cdot v\).

\text{(2) \(\Rightarrow\) (3):} Suppose \(\alpha, \beta \in F\). When \(\alpha \cdot v + \beta \cdot v =0\), we have \( \alpha \cdot v + \beta \cdot v  =0 \cdot v \in F \cdot v\).
Alternatively, if \(\alpha \cdot v + \beta \cdot v \in \operatorname{Span}(v)\backslash \{ 0 \} \), we get \(\alpha \cdot v + \beta \cdot v \in F \cdot v\) due to the equality \(F \cdot v = F \cdot (\alpha \cdot v + \beta \cdot v) \). Hence \(v \in Q(V)\).

\text{(3) \(\Rightarrow\) (1):} Let \(x \in \operatorname{Span}(v)\). Then, by Lemma \ref{spanlin}, \(x = \sum_{i=1}^{n} \alpha_{i} \cdot v\) for \(n \in \mathbb{N}\) and \(\alpha_{i} \in F\) where \(i \in \{ 1, \cdots ,n\}\). We proceed with an induction on \(n\) to establish \(\sum_{i=1}^{n} \alpha_{i} \cdot v= ({}^v\sum_{i=1}^{n} \alpha_{i}) \cdot v \). The base case \(n=1\) is straightforward. 
For the inductive step, assume that for some \(n \in \mathbb{N}\), \(\sum_{i=1}^{n} \alpha_{i} \cdot v= ({}^v\sum_{i=1}^{n} \alpha_{i}) \cdot v \) holds for all \(\alpha_i\in F\) where \(i \in \{ 1, \cdots, n \}\). Let \(\alpha_i\in F\) where \(i \in \{ 1, \cdots, n+1 \}\). We prove that \(\sum_{i=1}^{n+1} \alpha_{i} \cdot v = ({}^v\sum_{i=1}^{n+1} \alpha_{i}) \cdot v \). Given that \(v \in Q(V)\), we deduce that
\[
\sum_{i=1}^{n+1} \alpha_{i} \cdot v = \sum_{i=1}^{n} \alpha_{i} \cdot v + \alpha_{n+1} \cdot v = ({}^v\sum_{i=1}^{n} \alpha_{i}) \cdot v + \alpha_{n+1} \cdot v = \left({}^v\sum_{i=1}^{n}\alpha_{i} +_{v} \alpha_{n+1} \right) \cdot v,
\]
by the induction hypothesis. This completes the argument and establishes that \(x \in F \cdot v\), thereby proving that \(v\) is scalar.
\end{proof}

The following lemma provides justification for why the chosen notion of linear independence is appropriate for generalizing the concept outside the quasi-kernel.

\begin{lemma}\label{QL}
Let \(V\) be an \(F\)-space and \(S \subseteq V\). The following statements are equivalent.
\begin{enumerate} 
\item \(S\) is linearly independent; 
\item \(0 \notin S\) and for any non-empty \( A \finsub S\) and \(\alpha_a\in \operatorname{Span} (a)\) where \(a \in A\), if \(\sum_{a\in A} \alpha_a =0\), then \(\alpha_a=0\) for all \(a \in A\);
\item \(0 \notin S\) and \(\operatorname{Span}(S)= \oplus_{s\in S} \operatorname{Span}(s)\);
\item \(0 \notin S\) and for every subset \(T\) of \(S\), \(Span(T) \cap Span(S\backslash T)= \{ 0\}\);
\item \(0 \notin S\) and for every \(s \in S\), \(Span(s) \cap \operatorname{Span}(S\backslash \{s\})= \{ 0\}\);
\item \(0\notin S\) and \(\{ Span (s)\}_{s \in S}\) is additively independent. 
\end{enumerate}
\end{lemma}

\begin{proof} 
\((1) \Rightarrow (2)\Rightarrow (3) \Rightarrow (4) \Rightarrow (5) :\)
These implications are evident.

\((5) \Rightarrow (6):\)
If \(S= \emptyset\), then the result is clear. When \(S\) is not empty, we argue by contradiction. Suppose that \(\{\operatorname{Span}(s)\}_{s \in S}\) is linearly dependent. Then, there exists \(B \finsub S\) non-empty and \(\alpha_{b} \in \operatorname{Span}(s)\) for all \(b \in B\) such that \(\sum_{b \in B} \alpha_{b} = 0\), and the \(\alpha_{b}'s\) are non-zero. Let's say \(\alpha_{b_{0}} \neq 0\). It's important to note that \(\alpha_{b_{0}} \neq 0\) implies \(B\backslash \{b_{0}\} \) is non-empty. Then, \(\alpha_{b_{0}} = \sum_{x \in B\backslash \{b_{0}\}} (- \alpha_{x})\) with \(\alpha_{x} \in \operatorname{Span}(x)\) for all \(x \in B\backslash \{b_{0}\}\), and so \(\alpha_{b_{0}} \in \operatorname{Span}(S \backslash \{b_{0}\})\). This means \(\alpha_{b_{0}}\in \operatorname{Span}(S \backslash \{b_{0}\}) \cap \operatorname{Span}(b_{0})\), which contradicts our initial assumption. 

\((6) \Rightarrow (1):\)
If \(S= \emptyset\), then the result follows from the definition of linear independence. When \(S\) is not empty, suppose that there exists \(A \finsub S\) non-empty, \(n_{a} \in \mathbb{N}\), and \(\alpha_{a,i} \in F\) where \(a \in A\) and \(i \in \{1,\cdots , n_{a}\}\) such that \(\sum_{a \in A}\left(\sum_{i=1}^{n_{a}}\alpha_{a,i}\cdot a \right)=0\). Since \(\{ \operatorname{Span} (s)\}_{s \in S}\) is additively independent, it follows that \(\sum_{i=1}^{n_{a}}\alpha_{a,i}\cdot a = 0\) for all \(a \in A\).
\end{proof}

\begin{remark} \label{exa} 
\begin{enumerate} 
\item Clearly, by definition, for any \(v \in V\backslash\{0\}\), \(\{v\}\) is linearly independent.
\item Let \(S \subseteq V\).
In contrast to classical linear algebra, in near-vector space theory, \(t \notin \operatorname{Span}(S)\) is not equivalent to \(\operatorname{Span}(t) \cap \operatorname{Span}(S) = \{0\}\). For example, consider Example \ref{nearfield} (1). We have that \(\operatorname{Span}((1,0,1)) = \mathbb{R} \times \{0\} \times \mathbb{R}\) and \(\operatorname{Span}((0,1,1)) = \{0\} \times \mathbb{R} \times \mathbb{R}\). It's important to note that \((1,0,1) \notin \operatorname{Span}((0,1,1))\) since for any \((x,y,z) \in \operatorname{Span}((0,1,1))\), the \(x\)-coordinate is trivial. However, \((0,0,1) \in \operatorname{Span}((1,0,1))\) because
\[
(0,0,1) = \frac{\sqrt[3]{4}}{\sqrt[3]{3}} \cdot (1,0,1) - \frac{1}{\sqrt[3]{6}}\cdot (1,0,1) - \frac{1}{\sqrt[3]{6}}\cdot (1,0,1)
\]
and \((0,0,1) \in \operatorname{Span}((0,1,1))\) since 
\[
(0,0,1) = \frac{\sqrt[3]{4}}{\sqrt[3]{3}}\cdot (0,1,1) - \frac{1}{\sqrt[3]{6}} \cdot (0,1,1) - \frac{1}{\sqrt[3]{6}}\cdot (0,1,1),
\]
so 
\[
\operatorname{Span}((1,0,1)) \cap \operatorname{Span}((0,1,1)) \neq \{(0,0,0)\}.
\]
This illustrates that in near-vector space theory, the span behaves differently than in classical linear algebra. Specifically, we can decompose \(V\) as follows: \(W_1= \{ (x,y, 0) \mid x, y \in \mathbb{R}\}\) and \(W_2= \{ (0,0, z) \mid z \in \mathbb{R}\}\). \(W_1\) and \(W_2\) are \(F\)-subspaces of \(V\) and \(V=W_1\oplus W_2\). Additionally, \(Q(V)= W_1 \cup W_2\). Furthermore, if \(v \in W_i\) for some \(i \in \{ 1, 2\}\), then \(\operatorname{Span}(v)= F\cdot v \). It's clear that \(\operatorname{dim} (v)\) equals the minimal number of non-zero \(w_i\)'s when \(v = w_1+w_2\), where \(w_i \in W_i\) and \(i \in \{ 1, 2\}\). 
  \item   
  Given \(S \subseteq V\) and \(q \in Q(V)\), then \(q \notin \operatorname{Span}(S)\) is  equivalent to \(\operatorname{Span}(q) \cap \operatorname{Span}(S) = \{0\}\). 
    \item In (2), we have seen that there could be elements \(t\) in \(V\) outside of \(\operatorname{Span}(S)\) for some \(S\) subset of \(V\) and still have \(S \cup \{t\}\) be linearly dependent. However, if \(S\) is a linearly independent subset of \(V\) and \(q \in Q(V)\) with \(q \notin \operatorname{Span}(S) \), then \(\{q \} \cup S\) is linearly independent. This is immediately evident due to the direct sum property established in the statement (3) of Lemma \ref{QL}.
\item Let \(S \subseteq V\). It's worth noting that \(S\) being linearly independent is not equivalent to the following statement: for all \(A \finsub S\) non-empty, \(\alpha_a \in F\) where \(a \in A \), \(\sum_{a \in A } \alpha_a \cdot a =0\) implies \(\alpha_a =0 \) for all \(a\in A\). For instance, using the same example as in (2), consider the vectors \((1, 0, 1)\) and \((0, 0, 1)\). Suppose that for some \(\alpha, \beta \in \mathbb{R}\), we have 
\[
\alpha \cdot (1, 0, 1) +\beta \cdot (0, 0, 1) =(0, 0,0).
\]
Then, \(\alpha^3 + \beta^3=0\), and since \(\alpha=0\), we deduce that \(\alpha= \beta=0\). However, as seen in (2), \((0, 0, 1)\in \operatorname{Span} ((1, 0, 1))\), so the set \(\{ (1, 0, 1),(0, 0, 1)\}\) is not linearly independent. 
\item Let \(v \in V \backslash \{ 0\} \) and \(\Theta\subseteq Q(V)\) with \(|\Theta|=\operatorname{dim}(v)\) such that \(v = \sum_{q \in \Theta } q\). It follows that \(\Theta\) is linearly independent. This can be easily proved by contradiction.
\end{enumerate}
\end{remark}
The following result describes the span of elements in the quasi-kernel.

\begin{lemma}
\label{L=sum}
Let \(S \subseteq Q(V)\). Then \(\operatorname{Span}(S) = \sum_{s \in S} F\cdot s\). Moreover, the following statements are equivalent:
\begin{enumerate}
    \item \(S\) is linearly independent; 
    \item for any non-empty \( A \finsub S\) and \(\alpha_a\in F\) where \(a \in A\), if \(\sum_{a\in A} \alpha_a \cdot a =0\), then \(\alpha_a=0\) for all \(a \in A\)
    \item \(0 \notin S\) and \(\operatorname{Span}(S) = \bigoplus_{s \in S}F\cdot s\);
    \item \(0 \notin S\) and \( s\notin \operatorname{Span}(S\backslash \{s\})\);
    \item \(0 \notin S\) and \(\{ F\cdot s \}_{s\in S}\) is additively independent. 
\end{enumerate}
\end{lemma}

\begin{proof} 
Suppose \(S \subseteq Q(V)\). By Lemma \ref{scalar}, \(\operatorname{Span} (s)=  F\cdot s\) for any \(s\in Q(V)\). Therefore,  \(\operatorname{Span}(S) = \sum_{s \in S} F\cdot s\) follows from Lemma \ref{spanlin}, and the equivalence of statements \((1)\) to \((5)\) directly follows from Lemma \ref{QL}. 
\end{proof}

Let's now define a general notion of a generating set that allows us to consider sets outside of an \(F\)-subspace.

\begin{definition}
Let \(V\) be an \(F\)-space and \(W\) be an \(F\)-subspace of \(V\). A subset \(S\) of \(V\) is called a \textbf{generating set for \(W\)} if \(W \subseteq  \operatorname{Span}(S)\). 
\end{definition}

\begin{remark}
\begin{enumerate} 
\item Note that \(\emptyset\) is a generating set for \(\{0\}\). 
\item When \(S\subseteq W\), then \(S\) is a generating set of \(W\) if and only if \(\operatorname{Span}(S)=W\).
\item When \(V\) is a near-vector space, \(Q(V)\) is a generating set for \(V\).
\end{enumerate}
\end{remark} 

Now we can define two notions of a basis: the \(F\)-basis and the scalar \(F\)-basis. In classical linear algebra, these two concepts coincide. However when considering near-vector spaces, this distinction becomes fundamental.

\begin{definition}
\label{basis}
Let \(V\) be an \(F\)-space. \begin{enumerate} 
\item The set \(S\) is called an \textbf{\(F\)-basis for \(V\)} (or simply \textbf{basis} when there is no confusion) if \(S \subseteq V\), \(S\) is a generating set for \(V\), and \(S\) is  linearly independent. 
\item A \textbf{scalar \(F\)-basis} (or simply \textbf{scalar basis} when there is no confusion) is a set that is a basis whose elements are scalar. 
\item $V$ is said to be \textbf{finite-dimensional} if $V$ admits a finite scalar $F$-basis.
\item When \(V\) is a finite-dimensional near-vector space, we say that a basis \(B\) is a \textbf{minimal (resp. maximal) basis for \(V\)}, if \(|B|= \operatorname{min} \{ |C| \mid C \text{ basis of } V\}\) (resp. \(|B|= \operatorname{max} \{ |C| \mid C \text{ basis of } V\}\)).
\end{enumerate}
\end{definition}

\begin{remark}
Note that \(\emptyset\) is a scalar \(F\)-basis for \(\{0\}\). 
\end{remark}

 \section{Key span results}

The two main results of this paper are Theorem \ref{minimal} and Corollary \ref{spanlemma}. These results illuminate the fundamental principles of the theory of near-vector spaces. The significance of Theorem \ref{minimal} becomes evident in the proof of Theorem \ref{nvs}, establishing the proposition that any \(F\)-subspace is indeed a near-vector space. Meanwhile, Corollary \ref{spanlemma} plays a pivotal role in the demonstration of Theorem \ref{quotientnvs}, which asserts that a quotient by an \(F\)-subspace also constitutes a near-vector space. Notably, the initial establishment of Theorem \ref{nvs} was achieved within the framework of division rings in \cite{HowellMarques}. Subsequently, its validity was extended to general cases through a geometric approach in \cite{HR22}. Independently, in this paper, we present a direct generalization of the version in \cite{HowellMarques} to prove Theorem \ref{nvs}. 

Upon sharing our paper on arXiv, we became aware of Dr.~Wessels' work \cite[Theorem 3.3.7]{Wessels}. However, we wish to acknowledge that the proof of Theorem \ref{quotientnvs} as presented in \cite[Theorem 3.3.7]{Wessels} has been identified as incomplete. This issue has been acknowledged and agreed upon by the authors involved. Specifically, upon closer examination, Dr.~Marques identified a crucial missing element in their argument. This critical component is adequately provided by our Corollary \ref{spanlemma}, which essentially acts as the missing link required to achieve the completeness of the proof presented in \cite{Wessels}. We are also grateful to Prof.~Wadsworth for his insightful suggestions, which prompted us to reorganize the initial proof of Theorem \ref{minimal} and led us to incorporate the following lemma as part of our revised approach.
\begin{lemma}\label{mins}
Let $v \in V \backslash \{ 0\} $ and $\Theta\subseteq Q(V)$ with $|\Theta|=\operatorname{dim}(v)$ such that $v = \sum_{q \in \Theta } q$. For every $w\in \operatorname{Span} (v)$, we define $\gamma_{q,w} \in F$ for all $q\in \Theta$ such that $w = \sum_{q \in \Theta} \gamma_{q,w} \cdot q $ and $\Gamma_w :=\{\gamma_{q,w} \mid \gamma_{q,w} \neq 0, \ q \in \Theta \}$. Let $s \in \operatorname{Span} (v)$ such that $|\Gamma_s| = \operatorname{min}\{ |\Gamma_w| \mid w\in \operatorname{Span}(v)\backslash \{ 0\}\}$. Then,
\begin{enumerate}
    \item $s \in Q(V)$;
    \item $\operatorname{Span}(\Theta)= \operatorname{Span}( \{s\} \cup \Theta \backslash \{ q_0\})$, for all $q_0 \in \Theta $ such that $\gamma_{q_0,s} \neq 0$.
\end{enumerate}
\end{lemma}
\begin{proof}
Let $\Theta \subseteq Q(V)$ with $|\Theta|=\operatorname{dim}(v)$ such that $v = \sum_{q \in \Theta } q$ and $s\in \operatorname{Span} (v)$ such that $s = \sum_{q \in \Gamma_s} \gamma_q \cdot q $ and $|\Gamma_s| = \operatorname{min}\{ |\Gamma_w| \mid w\in \operatorname{Span}(v)\backslash \{ 0\}\}$.
\begin{enumerate}
    \item  We denote $|\Gamma_s|$ by $m$. If $m=1$, then $s\in Q(V)$. Consider $m \geq 2$. Suppose that $+_{\gamma_q \cdot q} = +_{\gamma_{q'} \cdot q'} $ for all $q, q' \in \Gamma_s$. Then $s \in Q(V)$ by Lemma \ref{pluses1}. Now, suppose that there exists $q_1, q_2 \in \Gamma_s$ such that $+_{\gamma_{q_1} \cdot q_1} \neq +_{\gamma_{q_2} \cdot q_2}$. Then, we can choose some $\alpha, \beta \in F$ such that $\alpha +_{ \gamma_{q_1} \cdot q_1} \beta \neq \alpha +_{ \gamma_{q_2} \cdot q_2} \beta$. Let $t = \alpha \cdot s + \beta\cdot s - (\alpha +_{\gamma_{q_1} \cdot q_1} \beta)\cdot s $, so that $t \in \operatorname{Span}(s) \subseteq \operatorname{Span}(v)$. Then 
$$t= \sum_{q \in \Gamma_s} \delta_q \cdot (\gamma_{q} \cdot q)$$
 where $\delta_q = ( \alpha +_{\gamma_q \cdot q} \beta) 
 -_{\gamma_q \cdot q} ( \alpha +_{\gamma_{q_1} \cdot q_1} \beta),$ for all $q\in \Gamma_s$.
 So, $\delta_{q_1}=0$ while $\delta_{q_2} \neq 0$. Hence, $t \neq 0$, since $\Gamma_s$ is linearly independent. This contradicts the minimality assumption. So, the final case cannot occur, proving (1).
\item By (1), we have $s\in Q(V)$. We set $\Theta':=\{s\} \cup \Theta \backslash \{ q_0\}$ where $q_0 \in \Theta $ such that $\gamma_{q_0} \neq 0$. We have 
$\operatorname{Span}(\Theta)= \operatorname{Span}(\Theta')$. Indeed, $\Theta' \subseteq \operatorname{Span}(\Theta)$ proving the inclusion $\operatorname{Span}(\Theta')\subseteq \operatorname{Span}(\Theta)$. In the other inclusion, we prove that $\Theta \subseteq \operatorname{Span}(\Theta')$. It is enough to prove that $q_0 \in \operatorname{Span}(\Theta')$, which is the case, since $q_0 = \gamma_{q_0}^{-1} ( s -\sum_{q \in \Theta \backslash \{q_0\}} \gamma_q \cdot q  )$.
 \end{enumerate}
\end{proof}

\begin{thm} 
\label{minimal}
Let $v \in V \backslash \{ 0\} $. There is $\Theta\subseteq Q(V)$ with $|\Theta|=\operatorname{dim}(v)$ such that $v = \sum_{q \in \Theta } q$ and  
\begin{equation*}
   \operatorname{Span}(v) = \operatorname{Span}(\Theta) =  \bigoplus_{q\in \Theta}F\cdot q.
\end{equation*}
\end{thm}

\begin{proof} 
We prove the theorem by induction on $\operatorname{dim}(v)$. When $\operatorname{dim}(v)=1$, then $v \in Q(V)$ and one can simply take $\Theta= \{v\}$. Now assume that $\operatorname{dim}(v)>1$. Suppose that the theorem is satisfied for any $v \in V\backslash \{ 0\} $ such that $\operatorname{dim}(v)<k$. We prove that the result remains true for any $v\in V\backslash \{ 0\} $ such that $\operatorname{dim}(v)=k$. Let $v \in V $ such that $\operatorname{dim}(v)=k$. By Lemma \ref{mins}, we can choose $\Theta \subseteq Q(V)$ such that $\Theta \cap \operatorname{Span}(v) \neq \emptyset$, $|\Theta|= \operatorname{dim}(v)$ and $\operatorname{Span}(v)= \operatorname{Span}(\Theta)$. Say $q_{0} \in \operatorname{Span}(v)\cap \Theta$. Since $v \in \operatorname{Span}(\Theta)$, we can write $v= \gamma_0 \cdot q_0 + w$ where $\gamma_0 \in F$ and $w\in  \operatorname{Span}(\Theta \backslash \{ q_0 \})$. Then, $\operatorname{dim}(w) \leq k-1 < \operatorname{dim}(v)$. The induction hypothesis on $\operatorname{dim}(v)$ yields that there is $\Theta_w \subseteq Q(V)$ with $|\Theta_w| = \operatorname{dim}(w)$ and $\operatorname{Span}( \Theta_w) = \operatorname{Span}(w)$. Let $\Theta_v =\{ q_0 \} \cup \Theta_w\subseteq Q(V)$. Then, $\operatorname{Span}( v)\subseteq \operatorname{Span}( \Theta_v)$, since $v = \gamma_0 \cdot q_0 + w$ with $q_0, w \in \operatorname{Span}( \Theta_v)$. For the reverse inclusion, we observe that $w\in \operatorname{Span}(v)$ since $w= v - \gamma_0 \cdot q_0$ and $q_0 \in \operatorname{Span}( v)$ by definition. So that $ \operatorname{Span}(w) \subseteq \operatorname{Span}( v)$. Now, since $\operatorname{Span}( \Theta_w) = \operatorname{Span}(w)$ and $q_0 \in \operatorname{Span}( v)$, we can deduce that $\Theta_v \subseteq \operatorname{Span}( v)$. Thus we obtain $\operatorname{Span}( v) = \operatorname{Span}( \Theta_v)$.
Finally, $|\Theta_v | = |\Theta_w |+1 =\operatorname{dim}( w)+1= \operatorname{dim}(v)  $. Since $\operatorname{Span}(\Theta) = \bigoplus_{q \in \Theta}F\cdot q$ (by Lemma \ref{L=sum}), the result follows.
\end{proof}

\begin{remark}\label{thebas}
Let $v \in V \backslash \{ 0\} $. By Theorem \ref{minimal}, there is $\Theta\subseteq Q(V)$ with $|\Theta|=\operatorname{dim}(v)$ such that $v = \sum_{q \in \Theta } q$ and $\Theta \subseteq Q(\operatorname{Span}(v))$. Moreover, $\Theta$ is a scalar basis for $\operatorname{Span}(v)$.
\end{remark} 

The following lemma is an adaptation of the classical linear algebra proof for near-vector spaces, which shows that any vector space admits a basis (see for instance \cite[Theorem 7.2.2]{Lightstone}). The proof will make use of Zorn's Lemma (see for instance \cite[Theorem 3.5.6]{Bloch}).

\begin{lemma}
\label{generatingset} Let $S\subseteq Q(V)$ be a generating set for $V$. Then there exists $B\subseteq S$ such that $B$ is a scalar basis.
\end{lemma}
\begin{proof}
When $V = \{0\}$, then $S = \emptyset$ is a basis for $V$. Suppose that $V \neq \{0\}$. Consider the set $\mathcal{L}$ of all linearly independent subsets of $S$. Then $\mathcal{L}$ is non-empty. Indeed, there is a non-zero element $s$ in $S$, since $V \neq \{0\}$, so that $\{s\} \subseteq S$ is a linearly independent set. Let $\mathcal{C}$ be a chain $L_{1} \subseteq L_{2} \subseteq \cdots \subseteq L_{i} \subseteq L_{j} \subseteq \cdots $ where $L_{i} \subseteq \mathcal{L}$ for each $i \in I$. We now prove that $\bigcup_{i \in I}L_{i}$ is an upper bound for $\mathcal{C}$. We have that $L_{i} \subseteq \bigcup_{i \in I}L_{i}$ for all $i \in I$. We want to show that $\bigcup_{i \in I}L_{i}$ is a linearly independent set. Since $0 \notin L_{i}$ for all $i\in I$, $0 \notin \bigcup_{i \in I}L_{i}$. Let $A=\{a_{1},\cdots, a_{n}\} \subseteq \bigcup_{i \in I}L_{i}$ with $A$ being non-empty, $n_{a} \in \mathbb{N}$ with $a \in A$, for all $\alpha_{a,k} \in F$ where $a \in A$ and $k \in \{1, \cdots, n_{a}\}$ such that $\sum_{a \in A} \left(\sum_{k=1}^{n_{a}}\alpha_{a,k} \cdot a \right) = 0$. Then for each $j \in \{ 1, \cdots, n\}$, we have $a_{j} \in L_{i_{j}}$ for some $i_{j} \in I$. Let $i_{0} = \operatorname{max} \{ i_{j} \mid j \in \{1,\cdots, n\}\}$. Then $ A \subseteq L_{i_{0}}$. Since $L_{i_{0}}$ is a linearly independent set, this implies that $\sum_{k=1}^{n_{a}}\alpha_{a,k} \cdot a = 0$ and so $\bigcup_{i \in I}L_{i}$ is linearly independent. Therefore $\bigcup_{i \in I}L_{i}$ is an upper bound for $\mathcal{L}$. Therefore, by Zorn's Lemma, $\mathcal{L}$ has a maximal element, say $B$. We prove that $B$ is a basis for $V$. To do this, it is enough to prove that $B$ is a generating set of $V$, since we already know that $B$ is linearly independent. We argue by contradiction. Suppose $B$ is not a generating set of $V$. Then there exists $v \in V$ such that $v \notin \operatorname{Span}(B)$. 
Thus, $\{v \} \cup B$ is linearly independent since $v \in Q(V)$, by Remark \ref{exa}, (4). Also, $B \subsetneq \{v\} \cup B$, contradicting the maximality of $B$. Therefore $B$ is a generating set of $V$, and $B$ is a scalar basis for $V$. 
\end{proof}

\begin{remark}
    \begin{enumerate} 
    \item Note that Lemma \ref{generatingset} is generally not true if $S$ is a generating set of $V$ that contains elements that are not scalar. For example, consider the near-vector space example given in Remark \ref{exa}, (2). Note that $\{(1,0,1),(0,1,1)\}$ generates $V$, since $\operatorname{Span}(\{(1,0,1),(0,1,1)\}) = V$. Indeed, we have $V \subseteq \operatorname{Span}(\{(1,0,1),(0,1,1)\})$ since $\{(1,0,0),(0,1,0),(0,0,1)\}$ is a scalar basis for $V$ and $\{(1,0,0),(0,1,0),(0,0,1)\} \subseteq \operatorname{Span}(\{(1,0,1),(0,1,1)\})$. However, $\{(1,0,1),(0,1,1)\}$ is not linearly independent, since $(0,0,1) \in \operatorname{Span}((1,0,1))\cap \operatorname{Span}((0,1,1))$ (as shown in Remark \ref{exa}, (2)). Since $(1,0,1)$ and $(0,1,1)$ do not generate $V$, no proper subset of $\{(1,0,1),(0,1,1)\}$ generates $V$. Hence, no subset of $\{(1,0,1),(0,1,1)\}$ is a basis of $V$.
    \item Not all $F$-bases have the same cardinality. Indeed, using the same example as in Remark \ref{exa}, (2), we have $\{ (1, 0, 0) , (0, 1,1)\}$ as a basis for $\mathbb{R}^3$ equipped with the scalar multiplication $\smallstar$. Also, $\{ (1, 0, 0), (0, 1,0) , (0, 0,1)\}$ is a scalar basis. However, these two bases do not have the same cardinality. 
    \end{enumerate}
\end{remark}
Since a near-vector space is generated by its quasi-kernel, the following corollary can be deduced from Lemma \ref{generatingset}.
\begin{corollary}\label{nvsbasis}
Every near-vector space admits a scalar basis. 
\end{corollary}
The subsequent lemma establishes that scalar bases are also maximal bases. This can be readily demonstrated, given that we have previously established in Theorem \ref{minimal} that any element in $V$ can be expressed as a direct sum of scalar elements.
\begin{corollary}
When \(V\) is a finite-dimensional near-vector space, any scalar basis for \(V\) is also a maximal basis. 
\end{corollary}


The ensuing result is an adaptation of the Replacement Theorem for vector spaces, tailored for near-vector spaces (see for instance \cite[Theorem 7.2.1]{Lightstone}).

\begin{lemma}
\label{replacement}
Let $S, T \finsub Q(V)$, where $S$ is linearly independent and $\operatorname{Span}(S) \subseteq \operatorname{Span}(T)$. Then, $|S| \leq |T|$ and there exists $T_0 \subseteq T$ such that $ \operatorname{Span}(S \cupdot T_0) = \operatorname{Span}(T)$, where $S \cupdot T_0$ denotes the disjoint union of $S$ and $T_0$.
\begin{proof}
Consider an element $s \in S$. Since $S \subseteq \operatorname{Span}(S) \subseteq \operatorname{Span}(T)$, it follows that $s \in \operatorname{Span}(T)$. Since $T\subseteq Q(V)$, as established in Lemma \ref{L=sum}, we can express $s = \sum_{t \in T} \alpha_{t}\cdot  t$, where $\alpha_{t} \in F$ for all $t \in T$. As $s \neq 0$, at least one of the $\alpha_{t}$'s must be non-zero; let's denote this as $\alpha_{t_{1}} \neq 0$ with $t_{1} \in T$. Consequently, we have $t_{1} = \alpha_{t_{1}}^{-1} \left(s - \sum_{t \in T \backslash \{t_{1}\}} \alpha_{t} \cdot t \right)$.
This implies $t_{1} \in \operatorname{Span}(\{s \}\cup T\backslash \{t_{1}\})$, which leads to $T \subseteq \operatorname{Span}(\{s\}\cup T\backslash \{t_{1}\})$. Therefore, $\operatorname{Span}(T) \subseteq \operatorname{Span}(\{s\}\cup T\backslash \{t_{1}\})$ by virtue of Remark \ref{smallest}. As $s \in \operatorname{Span}(T)$, we then have $\{s\}\cup T\backslash \{t_{1}\} \subseteq \operatorname{Span}(T)$. 
This implies that $\operatorname{Span}(\{s\}\cup T\backslash \{t_{1}\}) \subseteq \operatorname{Span}(T)$ according to Remark \ref{smallest}. Ultimately, we deduce $\operatorname{Span}(\{s\}\cup T\backslash \{t_{1}\}) = \operatorname{Span}(T)$.
By applying this process recursively to each element of $S$, we obtain $\operatorname{Span}(T) = \operatorname{Span} ( S \cup T')$, where $T' \subseteq T$ and $|T'|= |T|-|S|$. This yields $|S|\leq |T|$. By setting $T_0=T'\backslash (S \cap T')$, we ensure $\operatorname{Span}(T) = \operatorname{Span} ( S \cupdot T_0)$. 
\end{proof}
\end{lemma}

A direct consequence of Lemma \ref{replacement} is the following corollary.

\begin{corollary}\label{samecar}
Every scalar basis has the same cardinality. 
\end{corollary}

\begin{remark}\label{dimvtheta}
Let $v \in V\backslash \{ 0 \}$. For any $\Theta \subseteq Q(V)$ such that $\operatorname{Span}(v) = \oplus_{q \in \Theta}F \cdot q$, it follows that $|\Theta| = \operatorname{dim}(v)$. This conclusion can be drawn from the fact that $\Theta$ is both linearly independent and generates $\operatorname{Span}(v)$, thus making $\Theta$ a scalar basis for $\operatorname{Span}(v)$. By Corollary \ref{samecar}, we deduce that every scalar basis has the same cardinality, and the equality $|\Theta| = \operatorname{dim}(v)$ emerges from Theorem \ref{minimal} and Remark \ref{thebas}. In other words, for any scalar basis $\Theta$, $|\Theta| = \operatorname{dim}(v)$.
\end{remark}

With the aforementioned concepts established, a near-vector space can be understood as an abelian group endowed with a scalar group action that possesses a scalar basis.

\begin{thm}
\label{scalarbasis}
Let $V$ be an $F$-space. 
Then $V$ is a near-vector space over $F$ if and only if $V$ admits a scalar basis over $F$.
\end{thm}

\begin{proof} 
Assume $V$ is a near-vector space over $F$. Consequently, by Corollary \ref{nvsbasis}, $V$ possesses a scalar basis over $F$. Conversely, assume $V$ admits a scalar basis over $F$. Let $B= \{ b_i\}_{i \in I}$ be a scalar basis for $V$ over $F$. As a result, the scalar group action is free. Specifically, if $\alpha \in F$ and $v \in V$ such that $\alpha \cdot v = v$, we can deduce, via the definition of scalar basis and Lemma \ref{scalar}, that $B \subseteq Q(V)$. Additionally, there exists $C \finsub B$ with $\alpha_c \in F$ for all $c \in C$ such that $v = \sum_{c \in C} \alpha_c \cdot c$. Consequently, the equation $\alpha\cdot v = v$ can be rewritten as: $\alpha \cdot v = \alpha \cdot \left( \sum_{c \in C} \alpha_c \cdot c\right) = \sum_{c \in C} (\alpha\cdot  \alpha_c )\cdot c= \sum_{c \in C} \alpha_c \cdot c =v$. Equivalently, we have $\sum_{c \in C} (\alpha \cdot \alpha_c +_c (- \alpha_c))\cdot c=0$. This implies that $\alpha \cdot \alpha_c = \alpha_c$, as $C$ is a linearly independent set. This, in turn, is equivalent to $\alpha_c=0$ for all $c \in C $ or $\alpha=1$, given that $F\backslash \{ 0\} $ forms a group. Finally, in light of the scalar basis definition, it is clear that $V$ is additively generated by $Q(V)$. Hence, $V$ is a near-vector space over $F$. 
\end{proof} 

The following theorem provides a characterization for equality of spans:

\begin{thm}\label{dimspan}
Let $v \in V \backslash \{ 0\}$ and $w\in \operatorname{Span} (v)$. The following statements are equivalent:
\begin{enumerate} 
\item $\operatorname{Span} (v) = \operatorname{Span} (w)$; 
\item $\operatorname{dim} (v)= \operatorname{dim} (w)$;
\item For every scalar basis $\Theta $ of $\operatorname{Span} (v)$, if $w= \sum_{q \in \Theta}  \gamma_q \cdot q$ where $\gamma_q \in F$ for all $q \in \Theta$, then $\gamma_q\neq 0 $ for all $q \in \Theta$.
\end{enumerate}
\end{thm}

\begin{proof} 
$(1)\Rightarrow (2):$ The equivalence $\operatorname{dim} (v)= \operatorname{dim} (w)$ holds true due to all scalar bases of $\operatorname{Span} (v)$ having the same cardinality, as stated by Corollary \ref{samecar}. 

$(2)\Rightarrow (3):$ Let $\Theta $ be a scalar basis of $\operatorname{Span} (v)$ and suppose $w= \sum_{q\in \Theta} \gamma_q \cdot q$ where $\gamma_q \in F$ for all $q \in \Theta$. Since $|\Theta| = \operatorname{dim}(v)$, as per Remark \ref{dimvtheta}, $\gamma_q\neq 0 $ for all $q \in \Theta$, lest it contradict the definition of $\operatorname{dim} (w)$, given that $|\Theta|=\operatorname{dim} (v)= \operatorname{dim} (w)$.

$(3)\Rightarrow (1):$ Utilizing Theorem \ref{minimal}, we deduce the existence of $\Theta_v$ and $\Theta_w \finsub Q(V)\backslash \{ 0\}$, with $|\Theta_{v}| = \operatorname{dim}(v)$ and $|\Theta_{w}| = \operatorname{dim}(w)$, that allow us to express 
$ \operatorname{Span} (v) = \oplus_{q\in \Theta_v} F\cdot q$
and 
 $\operatorname{Span} (w) = \oplus_{q\in \Theta_w} F\cdot q$.
Since $\operatorname{Span} (w)  \subseteq \operatorname{Span} (v)$, there exists $\Theta \subseteq Q(V)\backslash  \operatorname{Span} (w)$, such that 
$$ \operatorname{Span} (v) = \left( \oplus_{q\in \Theta_w} F\cdot q \right) \oplus \left( \oplus_{q\in \Theta} F\cdot q \right)= \operatorname{Span} (w) \oplus \left( \oplus_{q\in \Theta} F\cdot q \right),$$
owing to Lemma \ref{replacement}. Consequently, $\Theta_w \cup \Theta$ and $\Theta_v$ form scalar bases for $\operatorname{Span} (v)$. This means that $w = \sum_{q\in \Theta_w\cup \Theta} \gamma_q \cdot q $, where $\gamma_q \in F$ for all $q\in \Theta_w\cup \Theta$. As $|\Theta_{w}| = \operatorname{dim}(w)$, it follows that $\gamma_q=0$ for all $q \in \Theta$. For assumption $(3)$ to be satisfied, $\Theta$ must be empty, and thus $\operatorname{Span} (v) =  \operatorname{Span} (w)$.
\end{proof} 

\begin{corollary}\label{division}
Let $v\in V\backslash \{0 \}$. Suppose that for all $w\in \operatorname{Span} (v)\backslash \{ 0 \}$, $(F, +_w , \cdot)$ is a division ring. Let $A \finsub F$. There is $\Theta \subseteq Q(V) \backslash \{0\}$ such that $\operatorname{Span} (v) =\oplus_{q \in \Theta} F\cdot v$ (see Theorem \ref{minimal}). Then $\operatorname{Span} (v) = \operatorname{Span} (\sum_{\alpha \in A} \alpha \cdot v)$ if and only if ${}^q \sum_{\alpha \in A} \alpha \neq 0$ for all $q\in \Theta.$ 

\begin{proof}
We notice that
$$\sum_{\alpha \in A} \alpha \cdot v= \sum_{\alpha \in A} \alpha \cdot \left( \sum_{q \in \Theta} q \right) = \sum_{q \in \Theta}  ({}^q \sum_{\alpha \in A} \alpha) \cdot q.$$
Because $|\Theta| = \operatorname{dim}(v)$, according to Remark \ref{dimvtheta}, if ${}^q \sum_{\alpha \in A} \alpha=0$ for some $q \in \Theta$, then $\operatorname{dim} (\sum_{\alpha \in A} \alpha \cdot v)< \operatorname{dim}(v)$. This implies $\operatorname{Span} (v) \neq \operatorname{Span} (\sum_{\alpha \in A} \alpha \cdot v)$, as indicated by Theorem \ref{dimspan}. Conversely, if ${}^q \sum_{\alpha \in A} \alpha\neq 0$ for all $q \in \Theta$, by the Decomposition Theorem (see \cite[Lemma 4.13]{Andre}), we can express $V = \sum_{i=1}^r V_i$ as the regular decomposition of $V$, where $V_i$ are the regular components of $V$ for $i\in \{ 1, \cdots, r\}$. As per \cite[Theorem 4.2]{HowellMarques}, for any distinct $q, q' \in \Theta$, there are distinct integers $i_q$ and $i_{q'}\in \{ 1, \cdots, r\}$ such that $q \in V_{i_q}$ and $q' \in V_{i_{q'}}$. Consequently, ${}^q \sum_{\alpha \in A} \alpha \cdot q \in V_{i_q}$. According to \cite[Theorem 4.5]{HowellMarques}, this results in $\operatorname{Span}(\sum_{\alpha \in A} \alpha \cdot v )=\oplus_{q \in \Theta} ({}^q \sum_{\alpha \in A} \alpha)\cdot q $. This leads to $|\Theta| =\operatorname{dim}(v) = \operatorname{dim} (\sum_{\alpha \in A} \alpha \cdot v ) $. Thus, $\operatorname{Span}(v)  = \operatorname{Span}\left(\sum_{\alpha \in A} \alpha \cdot v \right)$, by virtue of Theorem \ref{dimspan}, as desired.
\end{proof} 
\end{corollary}

From the proof of Corollary \ref{division}, we can observe certain interesting facts that are explained in the subsequent remark.

\begin{remark}
Let $v\in V\backslash \{0 \}$. Suppose that for all $w\in \operatorname{Span} (v)\backslash \{ 0 \}$, $(F, +_w , \cdot)$ is a division ring. 
\begin{enumerate} 
\item Let $A \finsub F$.  By Theorem \ref{minimal}, there exists $\Theta \subseteq Q(V) \backslash \{0\}$ such that $\operatorname{Span} (v) =\oplus_{q \in \Theta} F\cdot v$. As indicated by the proof of Corollary \ref{division}, we have 
$$ \operatorname{Span}\left(\sum_{\alpha \in A} \alpha \cdot v \right) = \oplus_{q \in \Theta_0} F \cdot q$$ where $\Theta_0 = \{ q \in \Theta \mid {}^q \sum_{\alpha \in A} \alpha \neq 0\} $. 
In particular, $ \operatorname{dim}(\sum_{\alpha \in A} \alpha \cdot v)= |\Theta_0|.$ 
\item By the proof of Corollary \ref{division}, $\operatorname{dim}(v)$ is always less than or equal to the number of regular components of $V$.
\end{enumerate}
\end{remark}

The subsequent corollary, derived from Theorem \ref{dimspan}, unveils the structural characteristics of the span formed by a linear combination of two elements.

\begin{corollary}
\label{spanlemma}
Let $v \in V \backslash \{ 0\}$ and $\alpha, \beta \in F$ such that $\alpha \neq \beta$. Then $\operatorname{Span}(v) =\operatorname{Span}(\alpha\cdot  v - \beta \cdot  v)$. Moreover, $\operatorname{dim}( v) = \operatorname{dim}(\alpha \cdot v - \beta \cdot v)$.
\end{corollary}

\begin{proof}
Consider a scalar basis $\Theta$ for $\operatorname{Span}(v)$. With Remark \ref{dimvtheta} implying $|\Theta|= \operatorname{dim} (v)$. We have $v = \sum_{q \in \Theta} \gamma_q \cdot q $, where $\gamma_q \in F$ for all $q \in \Theta$. For all $q \in \Theta$, $\gamma_q\neq 0$ due to the definition of $\operatorname{dim} (v)$. Consequently, $\alpha \cdot v - \beta\cdot v = \sum_{q \in \Theta} \delta_q \cdot q $, where $\delta_q = ( \alpha -_{\gamma_q \cdot q} \beta)\cdot \gamma_q $ for all $q\in \Theta$. Because $\gamma_q \neq 0 $ for all $q\in \Theta$ and $\alpha \neq \beta $, this implies $\delta_q \neq 0$. By virtue of Theorem \ref{dimspan}, we then deduce $\operatorname{Span}(v) =\operatorname{Span}(\alpha\cdot  v - \beta \cdot  v)$.
\end{proof}

In the following remark, we illustrate that when forming a linear combination of $v$ involving more than two terms, the span of $v$ might not match the span of this linear combination, even when the linear combination is non-trivial.

\begin{remark}
Building upon the same example as presented in Remark \ref{exa} (2), we can observe that
$$\operatorname{Span} ((0,1,1) ) \neq \operatorname{Span} \left(\frac{\sqrt[3]{4}}{\sqrt[3]{3}}\smallstar (0,1,1) - \frac{1}{\sqrt[3]{6}} \smallstar  (0,1,1) - \frac{1}{\sqrt[3]{6}}\smallstar  (0,1,1)\right) = \operatorname{Span} ((0,0,1) ) ,$$
even though $\frac{\sqrt[3]{4}}{\sqrt[3]{3}} +_{(0,0,1)} \left( -\frac{1}{\sqrt[3]{6}} \right) +_{(0,0,1)} \left(- \frac{1}{\sqrt[3]{6}}\right)= (\frac{4}{3} - \frac{1}{6}-\frac{1}{6})^{1/3} = 1 \neq 0$. This confirms that Corollary \ref{spanlemma} does not hold for a linear combination of more than $2$ terms. This discrepancy arises from the fact that the condition $\alpha\neq -\beta$ is sufficient to imply $\alpha +_q \beta= 0 $ for all $q \in Q(V)$, whereas such a condition is not available for three or more terms.
\end{remark}

\section{Proving the First Isomorphism Theorem for near-vector spaces}

We can now firmly establish the non-trivial result that any $F$-subspace, as defined at the beginning of the paper, is itself a near-vector space. This result emerges as a corollary of Theorem \ref{minimal}.

\begin{thm}
\label{nvs}
Let $W \subseteq V$. Then $W$ is an $F$-subspace of $V$ if and only if $W$ is a near-vector space over $F$ with the same operations as $V$.
\end{thm}

\begin{proof}
Let $W \subseteq V$. Suppose $W$ is an $F$-subspace of $V$. We only need to prove that the quasi-kernel of $W$ generates $W$ as an abelian group. Let $w \in W \backslash \{ 0\} $. Then by Theorem \ref{minimal}, there is $\Theta \finsub Q(V)$ with $|\Theta|= \operatorname{dim}(w)$ such that $w = \sum_{q \in \Theta } q$ and
$\operatorname{Span}(w) = \operatorname{Span}(\Theta)$. By Remark \ref{Q(W)subsetQ(W')} and Remark \ref{thebas}, $\Theta \subseteq Q(\operatorname{Span}(w)) \subseteq Q(W)$. 
Therefore, $Q(W)$ generates $W$ as an additive group. This proves that $W$ is a near-vector space with respect to the same operations as $V$. The converse is clear.
\end{proof}

We can also assert that any quotient of a near-vector space by an $F$-subspace is a near-vector space. This emerges as a corollary of Corollary \ref{spanlemma}.

\begin{thm}
\label{quotientnvs}
For any $F$-subspace $W$ of $V$, the quotient set $V/W$ is a near-vector space over $F$ with the operations induced from those of $V$. In particular, the quotient map $\pi : V\rightarrow V/W$ is a linear map with kernel $W$.
\end{thm}

\begin{proof}
As usual, we define the addition $\oplus$ and the scalar multiplication $\odot$ on $V/W$ as follows: For any $v, w \in V$, $\alpha \in F$, 
$$(v+W)\oplus ( w + W):= (v+w)+ W$$ 
and 
$$\alpha \odot (v + W) := (\alpha \cdot v) + W.$$ 
It is routine to prove that these operations are well-defined and induce on $V/W$ the structure of an $F$-space.

We need to prove that the scalar group action on $V/W$ is free, which means that if $\alpha \odot (v+W) = \beta \odot (v+W)$, then $\alpha = \beta$ or $v+W = W$. This is equivalent to stating that if $\alpha \cdot v - \beta \cdot v \in W$, then $\alpha = \beta$ or $v \in W$. Let $v \in V$, $\alpha, \beta \in F$ be such that $\alpha \cdot v - \beta \cdot v \in W$ and $\alpha \neq \beta$. Then $\operatorname{Span}(\alpha \cdot v - \beta \cdot v) \subseteq W$ and $v \in \operatorname{Span}(\alpha \cdot  v - \beta \cdot v)$ by Corollary \ref{spanlemma}. Hence, $v \in W$. 

It only remains to prove that $Q(V/W)$ generates $V/W$ as an abelian group. For this, it is enough to show that $\pi(Q(V)) \subseteq Q(V/W)$. Suppose $x \in \pi(Q(V))$. Then $x = \pi(v)$, where $v \in Q(V)$. Since $v \in Q(V)$, for all $\alpha, \beta \in F$, there exists $\gamma \in F$ such that $\alpha \cdot v + \beta \cdot v = \gamma \cdot v$. Because $\pi$ is a linear map, we have
    \begin{equation*}
        \alpha \cdot \pi(v) + \beta \cdot \pi(v) = \pi(\alpha \cdot v) +\pi(\beta \cdot v) = \pi(\alpha\cdot v+ \beta \cdot v) = \pi(\gamma \cdot v) = \gamma \cdot \pi(v).
    \end{equation*}
    As $V \subseteq V/W$, it follows that $\pi(v) \in Q(V/W)$. 
\end{proof}

\begin{remark}
    In the proof of Theorem \ref{quotientnvs}, we show that $\pi(Q(V)) \subseteq Q(V/W)$. This result can be generalized to any linear map. That is, given a linear map $f:V \rightarrow W$, we have $f(Q(V))\subseteq  Q(W)$.
\end{remark}

From the results above, we can deduce the classical correspondence between the kernel of linear maps and $F$-subspaces.  

\begin{cor}
\label{$F$-subspacekernel}
Let $W \subseteq V$. Then $W$ is an $F$-subspace of $V$ if and only if $W$ is the kernel of a linear map $\phi : V \rightarrow V'$.
\end{cor}

We can now conclude this section by stating the First Isomorphism Theorem for near-vector spaces. 

\begin{thm}\label{fit}
Let $f:V \rightarrow V'$ be a linear map. Then $V/\operatorname{Ker}(f) \cong \operatorname{Im}(f)$.
\end{thm}

\section{Near-vector spaces as modules over a ring}

In this section, we explore the concept of viewing a near-vector space as a module over a ring.

\begin{definition}
We introduce the notion of a \textbf{scalar group algebra} denoted by $\mathbb{Z}[F]$ as $\mathbb{Z}F/I$, where 
\begin{enumerate} 
\item $\mathbb{Z}F$ represents the semigroup algebra of the semigroup $F$. Specifically, $\mathbb{Z}F$ is a free $\mathbb{Z}$-module generated by the linearly independent set $\{ x_\alpha \}_{\alpha \in F}$. This algebra gains a $\mathbb{Z}$-algebra structure through the multiplication operation on $\mathbb{Z}F$, which is naturally induced by the group multiplication on $F$.
\item $I$ denotes the ideal of $\mathbb{Z}F$ generated by $x_0-0$. 
\end{enumerate}
An element of $\mathbb{Z}[F]$ is denoted as $\sum_{\alpha \in \Lambda}n_{\alpha}\cdot {\bf x}_{\alpha} $, where $\Lambda \finsub F$, $n_{\alpha} \in \mathbb{N}$, and ${\bf x}_{\alpha}$ is the equivalence class of $x_\alpha$ in the quotient $\mathbb{Z}F/I$, with $\alpha \in \Lambda$.
\end{definition}

\begin{remark}
We note that a scalar group algebra is a specific instance of a contracted semigroup ring.
\end{remark}
However, viewing a near-vector space as a module over a ring requires further consideration.

\begin{lemma} \label{module}
Every near-vector space $(V, \mu)$ over the scalar group $F$ naturally possesses a $\mathbb{Z}[F]$-module structure. This structure is defined by the action $\mu_{\mathbb{Z}[F]} : \mathbb{Z}[F] \times V \rightarrow V$, mapping $(\sum_{\alpha \in \Lambda}  n_\alpha \cdot {\bf x}_{\alpha}, v)$ to $\sum_{\alpha \in \Lambda} \mu (\alpha , n_\alpha \cdot v)$, where $\Lambda \finsub F$, and $n_\alpha \in \mathbb{N} $, with $\alpha \in \Lambda$.
\end{lemma}

\begin{remark}
It should be noted that for any $v\in V$, we have
$$\operatorname{Span}(v)= \operatorname{Span}_{\mathbb{Z}[F]}(v)$$ and for all $v \in Q(V)$,  $$\operatorname{Stab}_{\mathbb{Z}[F]}(v)= \left\{  \sum_{\alpha \in \Lambda} n_{\alpha} \cdot {\bf x}_{\alpha}  \mid {}^v \sum_{\alpha \in \Lambda } \big( {}^v\sum_{i=1}^{n_\alpha} \alpha \big)=1, n\in \mathbb{N}, \Lambda \finsub F, \alpha \in F \text{ and } n_{\alpha}\in \mathbb{N}, \forall a \in \Lambda \right\}.$$ 
\end{remark}

However, this perspective doesn't always induce a near-vector space structure over the scalar group $F$, unless certain conditions are met, such as having a scalar basis. To elaborate, consider the following lemma:

\begin{lemma} 
If $V$ is a $\mathbb{Z}[F]$-module, then it is possible to regard $V$ as an $F$-space. This is achieved through the scalar group action $\mu: F \times V \rightarrow V$, defined as $(\alpha, v) \mapsto \mu_{\mathbb{Z}[F]}({\bf x}_\alpha,v)$. 
\end{lemma} 

Let $\mu_{\mathbb{Z}[F]}$ be the module action induced by the scalar group action on the near-vector space $V$ as stated in Lemma \ref{module}.
We denote $\mathcal{L} = \{  F\cdot v \mid  v \in V \} $ as the set of "lines" passing through the origin. The module action $\mu_{\mathbb{Z}[F]}$ induces an action of $\mathbb{Z} [ F]$ on $\mathcal{L}$ as follows:
$$\mu_{\mathcal{L}} \left(\sum_{\alpha \in \Lambda} n_{\alpha} \cdot {\bf x}_{\alpha}, F\cdot v\right) :=   F\cdot \mu_{\mathbb{Z}[F]} \left( \sum_{\alpha \in \Lambda} n_{\alpha}\cdot {\bf x}_{\alpha},v\right),$$
where $\Lambda \finsub F$, and $n_\alpha\in \mathbb{N}$ for all $\alpha \in \Lambda$, and $v \in V$. 

This insight reveals that the elements within the quasi-kernel are capable of generating lines that maintain their linearity even when subjected to the $\mathbb{Z}[F]$-action $\mu_{\mathcal{L}}$ induced by elements that transform a non-zero vector $v$ into a non-zero element. Essentially, under this specific action of $\mathbb{Z}[F]$, the dimension of $v$ remains constant. These scalar elements serve as a robust basis for constructing a conventional geometric coordinate system. For instance, in Remark \ref{exa}, (2), the basis $\{ (1, 0,0), (0,1,0), (0,0,1)\}$ satisfies this criterion.

\end{document}